\documentclass{amsart}

\usepackage{amssymb,amsmath}
\usepackage{url}
\usepackage{amscd}
\usepackage{texdraw}
\usepackage[all]{xy}
\usepackage{fancyhdr}
\usepackage{amsmath}
\usepackage{graphicx}
\usepackage{subcaption}

\usepackage{amscd,color}
\usepackage[shortlabels]{enumitem}

\theoremstyle{plain}
\newtheorem{thm}{Theorem}
\newtheorem*{starthm}{Main Theorem}
\newtheorem{prop}[thm]{Proposition}

\newtheorem{definition}{Definition}

\newtheorem{remark}{Remark}

\input txdtools

\newcommand{\cala}{{\mathcal A}}
\newcommand{\calb}{{\mathcal B}}

\newcommand{\calf}{{\mathcal F}}

\newcommand{\calm}{{\mathcal M}}

\newcommand{\calr}{{\mathcal R}}
\newcommand{\cals}{{\mathcal S}}

\renewcommand{\AA}{{\mathbb A}}
\newcommand{\CC}{{\mathbb C}}
\newcommand{\DD}{{\mathbb D}}

\newcommand{\RR}{{\mathbb R}}

\newcommand{\ZZ}{{\mathbb Z}}

\newcommand{\bj}{{\mathbf{j} }}
\newcommand{\bJ}{{\mathbf{J} }}
\newcommand{\bjinf}{{\mathbf{j}_{\infty} }}
\newcommand{\bJinf}{{\mathbf{J}_{\infty} }}

\renewcommand{\hat}{\widehat}

\newcommand{\la}{\lambda}

\newcommand{\Log}{\mathrm{ Log}}

\begin{document}
\title[Accessible boundary points]{Accessible Boundary Points in  the Shift Locus of a Familiy of Meromorphic Functions with Two Finite Asymptotic Values}

\author{Tao Chen}
\address{LaGuardia Community College CUNY\\
Department of Mathematics\\
 31-10 Thomson Avenue \\ Long Island City, NY 11101, USA \\}
\email{tchen@lagcc.cuny.edu}

\author{Yunping Jiang}
\address{Department of Mathematics\\
 Queens College CUNY\\
   65-30 Kissena Blvd.\\
  Queens, NY 11367-1597, USA and   Mathematics Program,  Graduate Center CUNY\\
  365 Fifth Avenue\\
 New York, NY 10016, USA\\}
\thanks{This material is based upon work supported by the National Science Foundation. It is partially supported by a collaboration grant from the Simons Foundation (grant number 523341) and PSC-CUNY awards.}
\email{yunping.jiang@qc.cuny.edu}
\author{Linda Keen$^*$}
\address{  Mathematics Program, Graduate Center CUNY\\
  365 Fifth Avenue\\
 New York, NY 10016, USA\\}
\email{ lkeen@gc.cuny.edu; linda.keenbrezin@gmail.com}


\subjclass[2010]{Primary: 37F30, 37F20, 37F10; Secondary: 30F30, 30D30, 32A20}

\begin{abstract}  In this paper we continue the study,  began in \cite{CJK2},  of the bifurcation locus of a family of meromorphic functions with two asymptotic values,  no critical values and an attracting fixed point.  If we fix the multiplier of the fixed point, either of the two asymptotic values determines a one-dimensional parameter slice for this family.   We proved that the bifurcation locus divides this parameter slice into three regions, two of them analogous to the Mandelbrot set and one, the shift locus, analogous to the complement of the Mandelbrot set.  In \cite{FK, CK} it was proved  that the points in the bifurcation locus corresponding to functions with a parabolic cycle, or those for which some iterate of one of the asymptotic values lands on  a pole    are accessible boundary points of  the hyperbolic components of the Mandelbrot-like sets.  Here we prove these points, as well as   the points where some iterate of the asymptotic value lands on a repelling periodic cycle,  are also accessible from the shift locus.
\end{abstract}

\maketitle

\section{Introduction}
The investigation of the bifurcation locus in the parameter plane of quadratic polynomials where the dynamics is unstable has led to a lot of interesting mathematics and is still not completely understood.
In an early paper, [GK], part of the parameter space for the dynamics of the family $\calr_2$ of rational maps with one attractive fixed point and two critical values was shown to be similar to the parameter space of quadratic polynomials, although the existence of two varying singular values and poles made its structure  more complicated.   In particular, in a one-dimensional slice formed by fixing the multiplier of the fixed point,  the bifurcation locus separates the parameter plane into three regions,  one like the complement of the Mandelbrot set, and called the shift locus,  where both critical values are attracted to the same cycle and two complementary regions that are Mandelbrot sets containing stable, or hyperbolic components  where the critical values are attracted to different cycles.

  In this paper, we look at how the situation differs for the family
\[ \calf_2= \Big\{  f_{\la,\rho}(z)= \frac{e^z - e^{-z}}{\frac{e^z}{\la}- \frac{e^{-z}}{\mu}}, \, \, \frac{1}{\la} - \frac{1}{\mu}=\frac{2}{\rho} \Big\} \]
of meromorphic functions with two asymptotic values $\la$ and $\mu$,  no critical values and a fixed point at the origin whose multiplier $\rho$ lies in the punctured unit disk.
Some things are the same, of course.  In particular, stable dynamical behavior is always eventually periodic and controlled by the singular values.  There are, though, significant differences due to the maps in $\calf_2$ being infinite to one and to their branching over the singular values being logarithmic rather than algebraic.

In \cite{CJK2}, we studied the family $\calf_2$ using the  holomorphic dependence of the functions on two parameters, the multiplier $\rho$ and the asymptotic value $\la$, the other asymptotic value $\mu$ being a simple function of $\rho$ and $\la$.  We proved that, like $\calr_2$, if we take a slice by fixing $\rho$ in the punctured unit disk, the bifurcation locus in the resulting parameter plane again divides it into three distinct regions, one, a shift locus like the complement of the Mandelbrot set, where   both asymptotic values are attracted to the fixed point at the origin, and two complementary regions that are  Mandelbrot-like. They each contain  infinitely many hyperbolic  components  where the asymptotic values are attracted to different periodic cycles.    It had already been shown,  (see \cite{KK, FK, CK}), that each hyperbolic  component of the Mandelbrot-like sets is a universal cover of $\DD^*$ and that   the covering map extends continuously to the boundary.  Like the hyperbolic components of Mandelbrot set, the boundary contains points where the map has a parabolic cycle.  Unlike the Mandelbrot set, however, the hyperbolic components do not contain a ``center'' where the  periodic cycle contains the critical value and has multiplier zero.  Instead, they contain a distinguished boundary point with the property that as the parameter approaches this point, the limit of the multiplier of the periodic cycle attracting the asymptotic value is zero. It is thus called a ``virtual center''.   Virtual centers are also characterized by the property that one of the asymptotic values is a prepole, that is: some iterate lands on infinity and its orbit is finite.

In this paper, we are interested in the bifurcation locus in the slice of $\calf_2$ with $\rho$ fixed.  In particular, we characterize two subsets of points that are accessible from inside the hyperbolic components in the sense that  there is a curve in the hyperbolic component  whose accumulation set on the boundary consists only of that point.  In addition we prove that points in the bifurcation locus where the asymptotic value lands on a repelling periodic cycle are accessible from the shift locus.  In section~\ref{boundaries} we prove our main result:

\begin{starthm}\label{main} The parameters in the bifurcation locus that are virtual centers and the parameters for which the function $f_{\la,\rho}$ has a parabolic cycle or for which an asymptotic value is mapped onto a repelling cycle   by some iterate of $f_{\la, \rho}$ are accessible from inside the shift locus.
\end{starthm}

In other words,  the points with parabolic cycles and the virtual centers in the bifurcation locus are  accessible  both from inside  the hyperbolic components of the Mandelbrot-like sets and from the inside of the shift locus.

The first step in proving our results is to put a ``coordinate structure'' on the shift locus.   We showed in \cite{CJK2} that the shift locus is an annulus.   We  summarize that argument in section~\ref{Model}.  The discussion is similar to that for polynomials and  rational maps.   It uses quasiconformal mappings together with
 the dynamics of a fixed ``model function'' to characterize the shift locus by defining a ``Green's  function''  for the model.  This  function pulls back from the dynamic space of the model to the shift locus where it measures the relative rates of attraction of the asymptotic values to zero.  The inverse of the Green's  function defines level and gradient curves for those rates in the shift locus.

 The transcendental qualities of    $f_{\la,\rho}$ impart a much more complicated structure near the boundary of the shift locus than one has for rational maps.  We describe this structure first in the model.  As we did for rational maps in \cite{GK}, we start with a fixed level curve  of Green's  function and apply the dynamics of the model map.  In that case, there were two preimages of the curve but now there are infinitely many.

To understand the structure, we need first
 to identify each of the
infinitely many inverse branches of the function $f_{\la,\rho}$ with an integer.   The $n^{th}$ backward orbit of a point can then be assigned to a sequence of $n$ of these integers.  Applying $f_{\la,\rho}$ to the map acts as a shift map on the sequence.   As  for the Julia set of a rational map, the periodic points are assigned infinite periodic sequences.  Prepoles, which are now preimages of the essential singularity, correspond to  finite sequences.   Thus, assigning the ``integer" infinity to the point at infinity and taking the closure in the space of finite and infinite sequences of integers, we obtain a representation  of the Julia set of the model map with its dynamics by a sequence space that is compatible with the shift map.

We use this identification of the Julia set with the sequence space to construct paths in our model space,  and in section~\ref{shift} we transfer these paths  from the model to the shift locus.   The heart of the proof of the main theorem is to show that the paths in the shift locus have unique end points.

\section{Notation and basics}
\label{basics}

  Here we briefly recall  the basic definitions, concepts and notation we will use.  We refer the reader to standard sources on meromorphic dynamics for details and proofs.  See e.g. \cite{Berg, BF, DK2,KK,BKL1,BKL2,BKL3,BKL4}.

 We denote the complex plane by $\CC$, the Riemann sphere by $\hat\CC$ and the unit disk by $\DD$.  We denote the punctured plane by $\CC^* = \CC \setminus \{  0 \}$ and the punctured disk by $\DD^* = \DD \setminus \{  0 \}$.

 \medskip
To study the dynamics of a  family of meromorphic functions, $\{ f_{\la}(z) \}$,  we look at the orbits of points formed by iterating the function $f(z)=f_{\la}(z)$.   If $f^k(z)=\infty$ for some $k>0$, $z$ is called a pre-pole of order $k$ --- a pole is a pre-pole of order $1$.  For meromorphic functions, the poles and prepoles have  finite orbits that end at infinity.  The {\em Fatou set or Stable set, $F_f$} consists of those points at which the iterates form a normal family.   The Julia set $J_f$ is the complement of the Fatou set and contains all the poles and prepoles.

\medskip
If there exists a minimal $n$ such that  $f^n(z)=z$, then $z$ is called {\em periodic}.   Periodic points are classified by their multipliers, $\rho(z)=(f^n)'(z)$ where $n$ is the period: they are repelling if $|\rho(z)|>1$, attracting if $0< |\rho(z)| < 1$,   super-attracting  if $\rho=0$ and neutral otherwise.  A neutral periodic point is {\em parabolic} if $\rho(z)=e^{2\pi i p/q}$ for some rational $p/q$.  The Julia set is the closure of the set of repelling periodic points and is also the closure of the prepoles, (see e.g. \cite{BKL1}).

\medskip
A point $a$ is a {\em singular value} of $f$ if $f$ is not a regular covering map over $a$.
\begin{itemize}
\item    $a$ is a {\em critical value} if for some $z$, $f'(z)=0$ and $f(z)=a$.
\item    $a$ is an {\em asymptotic value} for $f$ if there is a path $\gamma(t)$ such that\\ $\lim_{t \to \infty} \gamma(t) = \infty$ and $\lim_{t \to \infty} f(\gamma(t))=a$;  $\gamma(t)$ is called an {\em asymptotic curve or an asymptotic path for $a$}.   
\item The {\em set of singular values $S_f$} consists of the closure of the critical values and the asymptotic values.  The {\em post-singular set is
\[P_f= \overline{\cup_{a \in S_f} \cup_{k=0}^\infty f^k(a)  \cup \{\infty\}}. \]}
For notational simplicity, if a prepole $p_n$ of order $n$ is a singular value, $\cup_{k=0}^{n} f^k(p_n)$ is a finite set with $f^{n} (p_n)=\infty$.
\end{itemize}
\medskip
A map $f$ is {\em hyperbolic} if $J_f \cap P_f = \emptyset$.

\medskip
In \cite{DK2} it is proved that every
 component of the Fatou set of a function with two asymptotic values and no critical values is  eventually periodic: that is,
  $f^n(D) \subseteq f^m(D)$ for some integers $n,m$.    In addition, the   periodic cycles of stable domains for these functions are classified there as follows:
\begin{itemize}
\item Attracting:  if the periodic cycle of domains contains an attracting cycle in its interior.
\item Parabolic: if there is a parabolic periodic cycle on its boundary.
\item Rotation: if $f^n: D \rightarrow D$ is holomorphically conjugate to a rotation map.  It follows from arguments in \cite{KK} that for maps with only two asymptotic values and no critical values,  rotation domains are always simply connected.  These domains are called {\em Siegel disks}.
\end{itemize}

\medskip
A standard result in dynamics is that each attracting or parabolic  cycle of domains contains a singular value. The boundary of each rotation domain is contained in the accumulation set of the forward orbit of a singular value.  (See e.g.~\cite{M1}, chap 8-11 or~\cite{Berg}, Sect.4.3.)

 By a theorem of Nevanlinna, \cite{Nev}, any meromorphic function with only two asymptotic values and no critical values   can be explicitly written as a linear transformation of the exponential map.  Therefore, putting the essential singularity at infinity and
 conjugating by an affine map we may assume the origin is a fixed point with multiplier $\rho$, and we may write $\calf_2$ as a family of  functions of the form
\[ \calf_2= \Big\{  f_{\la,\rho}(z)= \frac{e^z - e^{-z}}{\frac{e^z}{\la}- \frac{e^{-z}}{\mu}}, \, \, \frac{1}{\la} - \frac{1}{\mu}=\frac{2}{\rho} \Big\} \]
so that $\la$ and $\mu$ are the asymptotic values.
  Note that $f_{\la,\rho}(z)$ is not defined for $\la=0, \rho/2$.

The family $\calf_2$ depends on two complex parameters. We form a {\em dynamically natural slice} of $\calf_2$, in the sense of \cite{FK}, by fixing $\rho$, $|\rho|<1$,  and taking the asymptotic value $\la \in \CC \setminus \{0, \rho/2 \}$ as the parameter.  The other asymptotic value $\mu$ is then a simple function of $\rho$ and $\la$.  We write the functions in the slice as $f_{\la}=f_{\la,\rho}$.

Since the origin is an attracting fixed point, for every $\la \in \CC \setminus \{0, \rho/2 \}$, either $\la$ or $\mu=\mu(\la, \rho)$ is attracted by $0$.

\begin{definition}\label{calm} Let
\[ \calm_{\la}=\{\la \in \CC \setminus \{0, \rho/2 \} \,|\, \la \mbox{ is NOT attracted to the origin} \} ,\]
\[ \calm_{\mu}=\{\la \in \CC \setminus \{0, \rho/2 \}  \,|\,  \mu \mbox{  is NOT attracted to the origin} \} \]  and
\[ \cals =\{\la \in \CC \setminus \{0, \rho/2 \}  \,|\,  \la, \mu  \mbox{ are both attracted to the origin} \}.  \]
\end{definition}

We focus the discussion below on $\calm_{\la}$.   It is a summary of results in \cite{FK} and \cite{CJK2}.  We refer the reader to those papers for proofs. There is a completely analogous discussion for $\calm_{\mu}$.   See Figures \ref{lambdaplane} and~\ref{lambda-muplanelabel.pdf}.

 \begin{figure}
     \centering
  \includegraphics[width=5in]{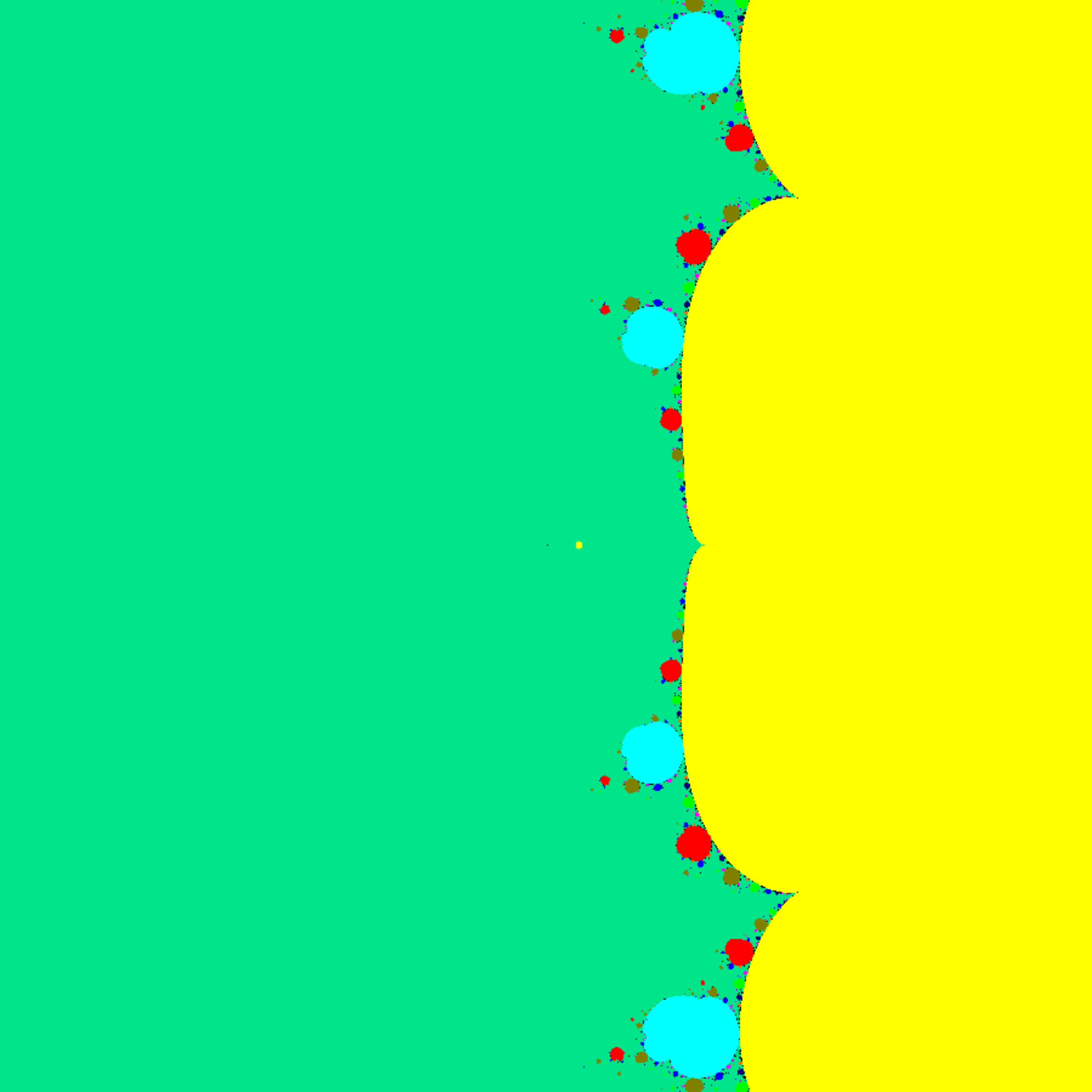}
  \caption{The $\la$ plane divided into the shift locus and shell components.  The shift locus is shown in green. In the yellow shell component, $\la$ is attracted to a fixed point; in the cyan components it is attracted to a period two cycle; in the red components, a period three cycle and in the khaki components a period four cycle. }
  \label{lambdaplane}
\end{figure}

 \begin{figure}
     \centering
  \includegraphics[width=5in]{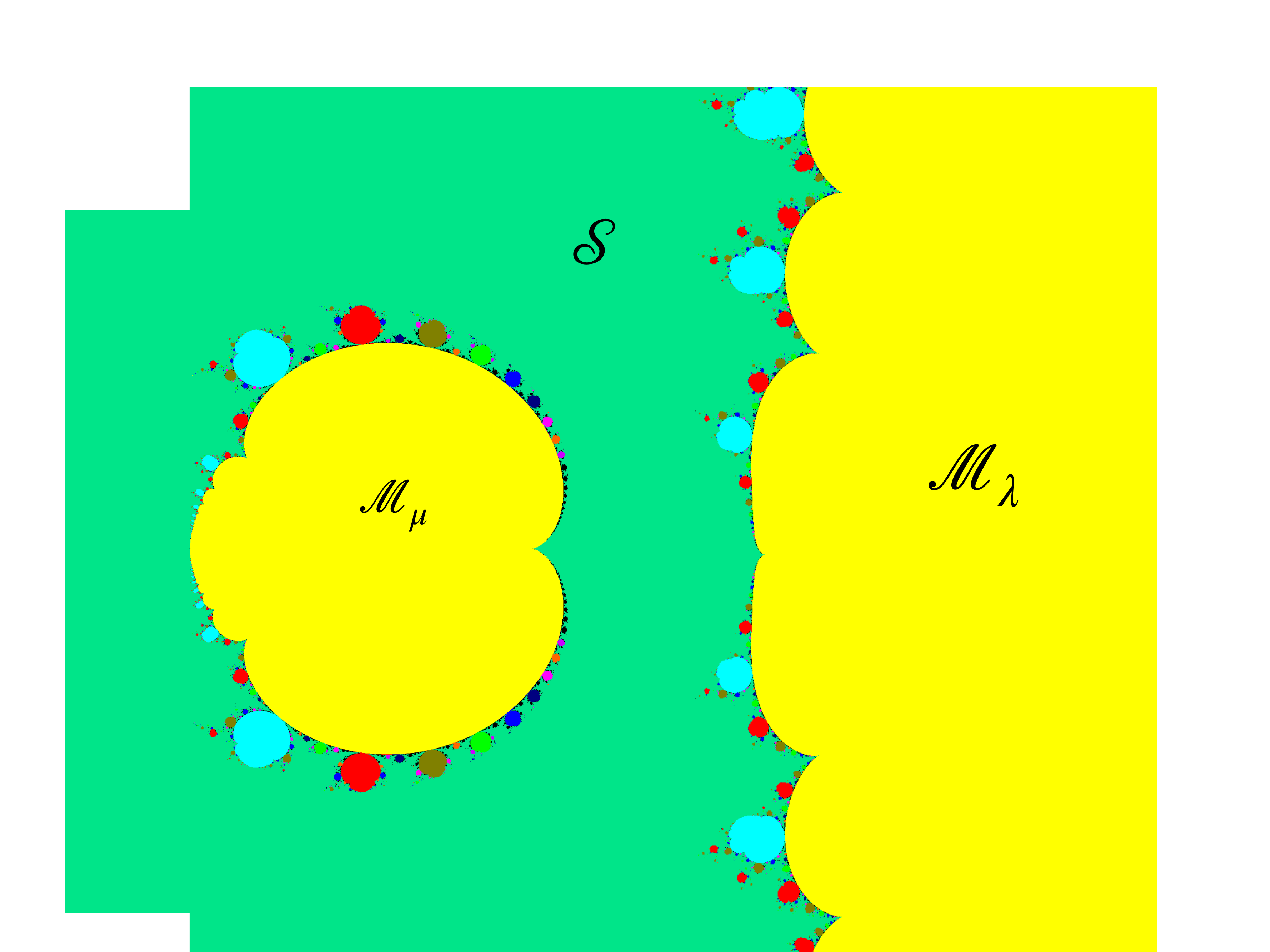}
  \caption{The $\la$-plane with a blow-up of the $\calm_{\mu}$ region. The color coding is the same as in figure~\ref{lambdaplane}.  }
  \label{lambda-muplanelabel.pdf}
\end{figure}

Recall that a function $f_{\la}$ is {\em hyperbolic} if the orbits of the asymptotic values remain bounded away from its Julia set.    The interior of $\calm_{\la}$  contains all the hyperbolic components $\Omega_p$ in which the orbit of $\la$ tends to an attracting periodic cycle of period $p$.  These are called {\em Shell components} because of their shape.   In \cite{FK} it is proved that each $\Omega_p$ is a universal covering of the punctured disk $\DD^*$ where the covering map is defined by the multiplier of the cycle.  This map extends continuously to the boundary and there is a standard bifurcation at each rational boundary point where the multiplier is of the form $e^{2p\pi i/q}$.  There is a unique  point $\la^* \in \partial\Omega_p$ such that as $\la \rightarrow \la^*$ along a path in $\Omega_p$, the multiplier of the cycle tends to zero.
  This boundary point is called the {\em virtual center} since it plays the role for $\Omega_p$ played by the center of a hyperbolic component of the Mandelbrot set for $z^2+c$.

For $p \neq 1$,   $\la^*$ is finite and has the additional property that $f_{\la^*}^{p-1}(\la^*)=\infty$; that is, $\la^*$ is a prepole.  Since the limit along a path approaching infinity from inside the asymptotic tract of an asymptotic value is the asymptotic value, $\la^*$,  its iterates, where the $p^{th}$ iterate is defined by this limit,  form a {\em virtual cycle}.  A parameter with this property is called a {\em virtual cycle parameter} and in \cite{FK} it is proved that every virtual center parameter on the boundary of a shell component is a virtual cycle parameter.

\begin{remark}\label{shellbdy} We note that all the rational boundary points and the virtual center of a shell component $\Omega_p$ are accessible boundary points   in the sense that the accumulation set of any path inside $\Omega_p$ tending to the point consists of a single point. \end{remark}

 There is a unique component $\Omega_1$ in which $\la$ is attracted to a fixed point $q_{\la} (\neq 0)$.   It is unbounded and, by abuse of language, we say its virtual center is infinity.

\section{The Model Map}\label{Model}
We fix $\rho=\rho_0, |\rho_0|<1$, once and for all for this paper and write $f_{\la}$ for $f_{\la, \rho_0}$.  The figures in the paper were made by taking $\rho_0=2/3$.

Let  $\la_0$ be a point in $\Omega_1 \subset \calm_{\la}$ such that at the fixed point $q_0=q(\la_0)$ of $f_{\la_0}$,  $f_{\la_0}'(q_0)=\rho_0$.  We remark that $\la_0$ is not uniquely defined because there is  a discrete set of preimages of $\rho_0$ under the multiplier map.    Choose one and set $Q(z)=f_{\la_0}(z)$; it is our {\em Model map}.\footnote{Note that conjugating by the affine map $w=z-q_0/2$ we obtain a map of the form $\alpha \tan w$ for some $\alpha$, $|\alpha|>1$, with fixed points at $\pm q_0/2$.   In particular, if $\rho_0$ is real, $\la_0$ can be chosen real,  and then the attracting basin of $q_0$ is a  half plane and the Julia set is a line.}
 Let $K_0$ be the immediate attracting basin of $q_0$ for $Q$.

   Since the orbits of the asymptotic values either tend to $q_0$ or $0$, the map $Q$ is hyperbolic.   Its Julia set $J_0$ is the common boundary of the attracting basins of $0$ and $q_0$;  it is well known to be equal to  the closure of both the set of  repelling periodic points of $Q$ and the set of prepoles of $Q$.   It will be convenient to identify both these sets with a space of sequences.   To that end we define

\begin{definition}\label{Sigma} Let  $\bj_n= j_1 j_{2} \ldots j_{i}\ldots j_n$,  $j_i \in \ZZ$ be a sequence of length $n$ whose entries are integers and set $\bJ_n= \{ \bj_n \}$;   let $\bjinf=j_1 j_{2} \ldots j_{i}\ldots j_n\ldots$,  $j_i \in \ZZ$ be a sequence of infinite length whose entries are integers and set $\bJinf= \{ \bjinf \}$.   Let $\bj$ denote an element of either $\bJ_n$ or $\bJinf$.  Define the sequence space
\[ \Sigma = \{ \bj \in  \bigcup_n \bJ_n  \cup \bJinf \cup \{ \infty \} \}  \]
and give it the standard sequence topology.  The shift map $\sigma$ defined by dropping $j_1$ defines a continuous self-map on $\Sigma$.
\end{definition}

 We will use $\Sigma$ to label the inverse branches of $Q$. (See Figure~\ref{ModelA}.)

 Let   $l$ be a curve  that joins   $\la_0$ to $\infty$   in $K_0$;  it is an asymptotic curve for $\la_0$.  Set      $l_*= \bigcup_{n \geq1} Q^n(l)$; this union is a continuous curve with endpoints $q_0$ and $\la_0$.   Although it does not matter in this section, we will see in the next section that we can choose $l$ so that it depends holomorphically on $\rho_0$.

  Since $\mu_0 \not\in \overline{K}_0 $
 we can define the inverse branches $R_j$ of $Q$, $j \in \ZZ$,   on $K^*= \overline{K}_0 \setminus   l$.
  Note that $Q$ is periodic with period $\pi i$ and so if  $q_j=q_0+ j \pi i$, then $Q(q_j)=q_0$;  it will be convenient to choose $l$ so that each $R_j$ is defined in a full  neighborhood of $q_0$ and is labeled so that $R_j(q_0)=q_j$.  Our figures are computed with $\rho_0=2/3$ so that $\la_0$ is real and $l$ is contained in the real axis.

  We can do this by setting
 \[  R_j(w) = \frac{1}{2}\Log \big(\frac{w/\mu_0 -1}{w/\la_0 -1} \big) + \pi i j  \]
 where $\Log$ stands for the principal branch of the logarithm.

Since $Q$ is periodic, we can speak of adjacent poles in $J_0$.   Let $l^+$ and $l^-$ mark the upper and lower edges of the curve $l$.  Then $R_j$ maps $K_0 \setminus l$ to a strip of width $\pi$ in $K_0$ bounded by the curves
$R_j(l^-)$ and $R_j(l^+)$, each of which joins a pole in $J_0$ to infinity; moreover, the poles are adjacent.  We label these poles $p_j$ and $p_{j+1}$ respectively.   Specifically,
  for $w \in K^*$,  we set
 \[ p_j =\lim_{w \to \infty \mbox{ asymptotic to  } l^-} R_j(w) = \frac{1}{2} \Log  \frac{\la_0}{\mu_0}  +  \pi i j \] and
  \[ p_{j+1} =\lim_{w \to \infty \mbox{  asymptotic to } l^+} R_j(w) = \frac{1}{2} \Log  \frac{\la_0}{\mu_0}  +  \pi i (j+1) \]
  Note that although each pole is defined by two limits, the index of the pole is well defined.  (See Figure~\ref{ModelA}.)

 Taking one sided limits   we define the images of $l$  under $R_j$ to be the lines
 \[ l_j= \{ R_j(w) \quad |\ w \in l^- \}. \]

\begin{enumerate}
\item  Set
 \[ R_{\bj_n}= R_{j_1j_{2} \ldots j_n}=R_{j_1} \circ R_{j_{2}}\circ \ldots \circ R_{j_n}.  \]
 \item  We can enumerate the preimages of the fixed points in the same fashion.   We set
 \[ q_{\bj_n}= q_{j_1j_{2} \ldots j_n}=R_{\bj_n}(q_0). \]
\item
We can enumerate the prepoles.    Set
\[ p_{\bj_n}= p_{j_1j_{2} \ldots j_n}=R_{\bj_n}(\infty). \]
This scheme enumerates all the prepoles by assigning each a unique element of $\Sigma$.  Note that
\[  Q(p_{j_1j_{2} \ldots j_n})=p_{j_{2} \ldots j_n} \]
so that $Q$ acts as a shift map $\sigma$ on the sequence.
\item
We can also enumerate the repelling periodic points in $J_0$.   If $z \in J_0$ and $Q^n(z)=z$, we can find branches $R_{j_k}$ such that
\[ z=R_{j_1j_{2} \ldots j_n}(z)= R_{j_1j_{2} \ldots j_n}R_{j_1j_{2} \ldots j_n}(z)= \ldots R_{j_1j_{2} \ldots j_n}(z). \]
Or
\[ z=R_{\bj_n}(z)= R_{\bj_n}R_{\bj_n}(z)= \ldots R_{\bj_n} R_{\bj_n}(z). \]
That is, we can associate the infinite repeating sequence $\bjinf=\bj_n \bj_n \ldots \bj_n \ldots $ to the point and set
\[ z= z_\bjinf =R_{\bj_n}(z_{\bjinf} )=Q^n(z_{\bjinf}).\]
Again $Q$ acts as a shift map and $Q^n$ leaves the infinite sequence invariant.

\end{enumerate}

\begin{prop}\label{Prop J0}  There is a homeomorphism from $\Sigma$ to $J_0$ such that for $\bj \in \Sigma$ and $z \in J_0$, $Q(z_{\bj} )= z_{\sigma(\bj)}$.
 \end{prop}

In \cite{CJK2}, using results in \cite{FK},  we  proved that the finite sequence defining the prepole that is the  virtual cycle parameter of a component $\Omega_p$ of $\calm_{\la}$ uniquely characterizes that component.  Thus, the finite sequences in $\Sigma$ are in one to one correspondence with these boundary points of  $\calm_{\la}$.

\medskip

\subsection{A structure for $K_0$}

 There is a linearizing homeomorphism $\phi_0$, defined in a largest neighborhood $\Delta=O_{\la_0}$ of $q_0$ to a disk $\DD_0$ centered at the origin of radius $r_0$,  such that  $\phi_0(q_0)=0$, $\phi_0'(q_0)=1$.  Moreover,    for $z \in \Delta$, $\phi_0(Q(z))=\rho_0 \phi_0(z)$,  $\la_0 \in \partial \Delta$, $\phi_0$ extends continuously to the boundary and $r_0 = |\phi_0(\la_0)|$. The function $\log|\phi_0|$ is like a Green's function for $\Delta$. The preimages of the circles $|\zeta|=r$ in $\DD_0$ are the {\em level curves} and the preimages of the radii are the {\em gradient curves} in $\Delta$.  Thus the level of $\partial\Delta$ is $r_0$.

The map $\phi_0$ depends holomorphically on $\rho_0$.   Therefore a canonical choice for the curve $l$ used to define the branches $R_j$ above can be made by  using the polar coordinates of $\DD_0$.   For example, if $\rho$ is real, we define $l_*=\phi_0^{-1}(t)$, $t \in [0, r_0) \in \DD_0$ and let $l=R_0(l_*) \setminus l_*$.  If it is not, we adjust appropriately.

 Since $\mu_0 \not\in K_0$, the map $\phi_0$ can be extended by analytic continuation to all of $K_0$.  The curves $R_0^{-1}(\partial\Delta)$ have level $r_0/\rho_0$ and end at infinity.   We use the level and gradient curves to define a coordinate structure on $K_0$ as follows.  The coordinates are locally defined on preimages of $A_0=R_0(\Delta) \setminus ( \Delta \cup l)$ and $\mathbf{B_0}=R_0(\overline{\Delta} \setminus \{\la_0\})$.

 \subsubsection{Fundamental domains}

 \begin{definition}  We say that a region $D \subset K_0$,  with  interior $D_0$, is a {\em fundamental domain} for the action of $Q$ if
\begin{itemize}
\item     for any pair $(z_1, z_2) \in D_0, \, z_1 \neq z_2, \, , Q(z_1) \neq Q(z_2)$ and if
\item for some integer $n \in \ZZ$,  $\cup_{k \geq n}^{\infty} \overline{Q^k(D_0)} = \overline\Delta$.
\end{itemize}
\end{definition}
We now inductively define a set of fundamental domains that defines a partition of $K_0$ into fundamental domains.  Each fundamental domain $D_0$ will have boundary curves that are identified by $Q$.  In the process we give an enumeration scheme for these domains.   The curves referred to in the description are shown in  figure~{\ref{ModelTess}.

 \begin{figure}
  \centering
  \includegraphics[width=5in]{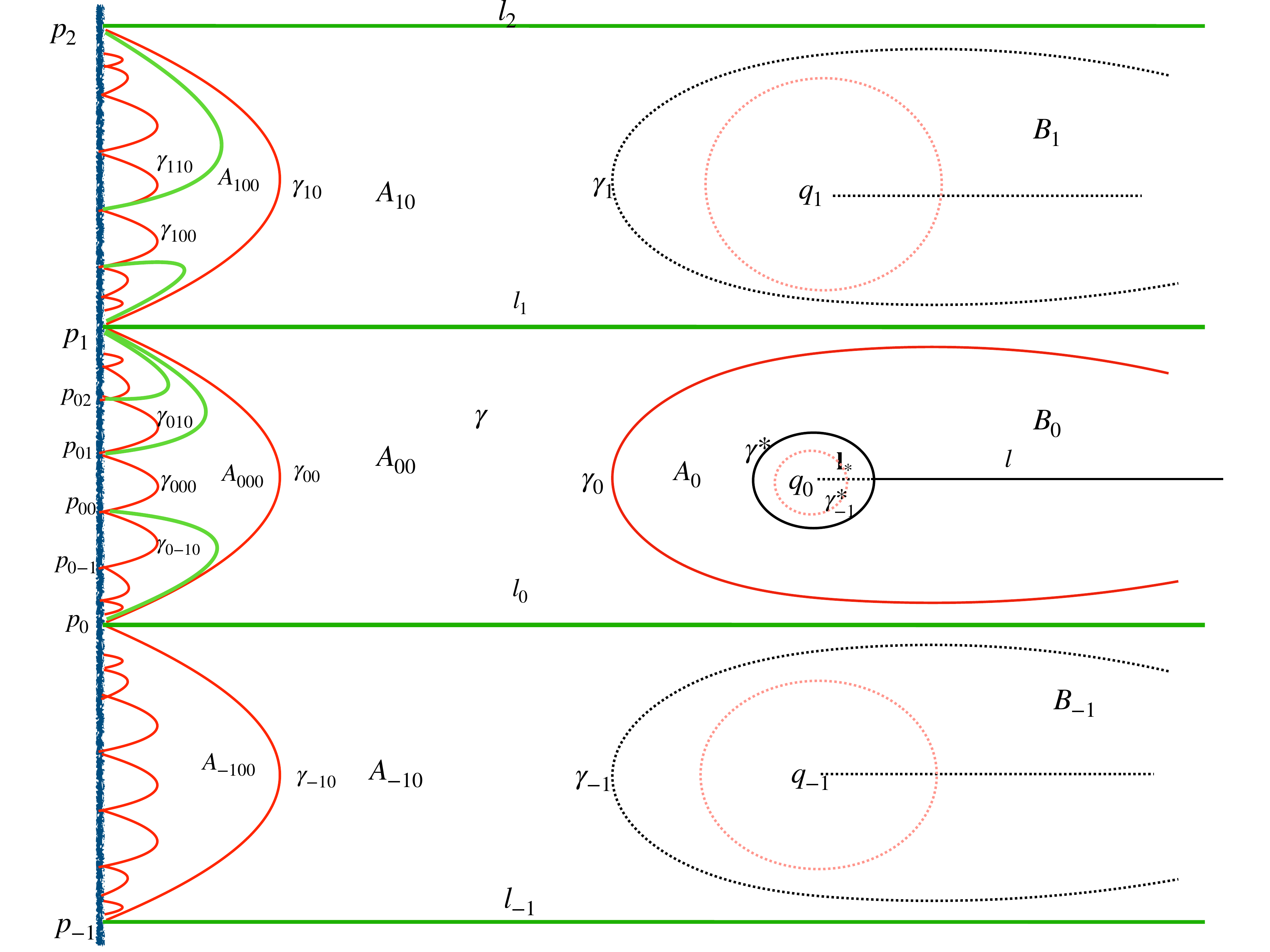}
  \caption{The model space shown with level (red)  and gradient (green) curves.  The dotted black curves are also level curves. }\label{ModelTess}
\end{figure}

\begin{enumerate}
\item  Let $\gamma_0^*=\gamma^*= \partial \Delta$.  Then $Q(\Delta) \subset \Delta$ and $\Delta$ contains the positive orbit of $\la_0$.   Set   $\gamma_{-n}^* = Q^n( \gamma^*)$,  $n=1, \ldots, \infty$.  Note that $R_0^n(\gamma_{-n}^*)= \gamma^*$.  These curves are nested  about $q_0$ in $\Delta$.  Any annulus $A_{-n}^*$ in $\Delta$  bounded by  $\gamma_{-n}^*$ and $\gamma_{-n-1}^*$ is a fundamental domain for $Q$.   In figure~\ref{ModelTess}, $\gamma^*$ is drawn in black and $\gamma^*_{-1}$ is drawn as a dotted red curve.  The annulus $A_0^*$ between them is a fundamental domain.   
\item
 For all $j \in \ZZ$, set $\mathbf{B}_j=R_j(\overline{\Delta} \setminus \{\la_0\})$.   In figure~\ref{ModelTess}, we see that  the domains $\mathbf{B}_j$, $j \neq 0$, are simply connected and  are bounded by dotted doubly infinite black curves $\gamma_j=R_j(\gamma^* \setminus \{\la_0\})$;   the dotted red curves inside the $\mathbf{B}_j$ are $R_j(\gamma_{-1}^*)$; each $\mathbf{B}_j$ contains the preimage of the fixed point $q_j$.  We single out the boundary curve of the domain $\mathbf{B}_0$, $\gamma_0=R_0(\gamma^* \setminus \{\la_0\})$ and color it red because, as we will see,      its preimages are somewhat different from the preimages of the other $\gamma_j$'s.    Note that $\mathbf{B}_0$ has a puncture at $\la_0$.  The annulus $A_0 = \mathbf{B}_0 \setminus  ( \overline{\Delta} \setminus \{\la_0\})$ between $\gamma_0$ and $\gamma^*$ is a fundamental domain, and we will be particularly interested in its preimages.  Note that it contains the line $l$.

\item Next, we denote the preimages of $A_0$ by $A_{j0}= R_j(A_0)$, $j \in \ZZ$.  To see what they look like, we look at the preimages of its boundary curves. 
First consider the boundary curve $\gamma^*$:  its preimages are the curves  $\gamma_j=R_j(\gamma^*)$.  The other boundary curve is $\gamma_0$: its preimages are the curves $\gamma_{j0}=R_j(\gamma_0)$; each of these curves, drawn in red in figure~\ref{ModelTess}, joins the pair of poles $(p_j, p_{j+1})$.   Now 
consider the preimages of the line $l$ inside $A_0$:  these are the lines $l_j=R_j(l^{-})$  and $l_{j+1}=R_j(l^{+})=R_{j+1}(l^{-})$ that join the poles $p_j$ and $p_{j+1}$ to infinity; they are drawn in green in figure~\ref{ModelTess}.     
Thus we see that each $A_{j0}$ is bounded by four curves, the red curves $\gamma_j$, $\gamma_{j0}$ and the green curves  $l_j$ and $l_{j+1}$.

Each $A_{j0}$ has three vertices on $\partial{K_0}$,  the  poles $p_j$ and $p_{j+1}$ where the lines $l_j$ and $l_{j+1}$ meet 
 the ends of $\gamma_{j0}$,  and infinity,  
the common endpoint of the doubly infinite $\gamma_{j}$ and an endpoint of each of the lines $l_j$ and $l_{j+1}$.

 Note that if we consider a pole as a prepole of order one, and  infinity as a ``prepole'' of order zero, the red curves labeled  by $\gamma_{i}$ join a prepole  to a prepole of the same order, whereas  the green curves labeled by $l_{i}$ join a prepole of order 0 to a prepole of order $1$.

 For later use  we set $\mathbf{A}_0=\overline{ \cup_j A_{j0}}$;   it is a simply connected domain.

 \item  We inductively define the domains $A_{j_{1} \ldots j_{n-1}0}=R_{j_1}(A_{j_{2} \ldots j_{n-1}0})$.  Their boundaries contain two pairs of  boundary curves.  The pair drawn in  red, consists of  $\gamma_{j_{1} \ldots j_{n-1} 0}$ which joins adjacent prepoles of order $n-1$, and  $\gamma_{j_{2} \ldots j_{n-1}0}$ which joins  a prepole of order $n-2$ to itself if $j_{n-1} \neq 0$ and joins two adjacent prepoles of order $n-2$ if $j_{n-1}=0$.  The other pair, drawn in green,  consists of  $l_{j_{1} \ldots j_{n-2} j_{n-1}0}$ and $l_{(j_{1}+1)j_{2} \ldots j_{n-2} j_{n-1}0}$, each joining a prepole of order $n-2$ to a prepole of order $n-1$.  If $j_{n-1}=0$, all four boundary prepoles are   distinct; if not, the prepoles of order $n-1$ are the same.  
 
 The domains $A_{000}$ and $A_{100}$ are shown in figure~\ref{ModelTess}.  Although not all the curves are labeled, the red   curves are preimages of $\gamma_0$ and the green curves are preimages of $l$.

Again, for later  use,    we define the simply connected domains
\[ \mathbf{A}_{ j_{1} \ldots j_{n-1}0}=R_{j_1}(\mathbf{A}_{j_{2} \ldots j_{n-1}0} ) \]
for all $n>1$. 
 
 \item   Unlike the preimages of $\gamma_0$ which join two adjacent poles, the curves  $\gamma_{ji}=R_j(\gamma_i)$, $i \neq 0$ are curves that join a pole  to itself.  This is  because of the way we have defined the $R_i$.   The pole $p_j$ is one endpoint of each the curves $\gamma_{(j-1) 0}$ and $\gamma_{j0}$.  The curves $\gamma_{ji}$,  $i>0$, are loops that come in to $p_j$, tangent to, and under $\gamma_{j0}$ while the curves $\gamma_{(j-1) i}$, $i<0$,   are loops that come in to $p_j$, tangent to, and under $\gamma_{(j-1)0}$.  Thus, in the preimage,  $R_i(K_0)$, the curves $\gamma_{ji}$ are tangent to the pole $p_{i}$ for $j \geq 0$ and tangent to the pole $p_{i+1}$ for $j \leq 0$.   
 
 Therefore the loops $\gamma_{ji}$, $j \neq 0$, bound simply connected domains $\mathbf{B}_{ji}=R_j(\mathbf{B}_i)$.   They are tesselated by the fundamental annuli $B_{ji}^k=R_{ji}(A_{-k}^*)$,   $k \geq 0$ and each  $\mathbf{B}_{ji}$ contains a curve that is a preimage of the line $l_*$.  The  disjoint domains $\mathbf{B}_{ji}$  form an infinite cluster  at each pole.   For considerations of space and clarity, we haven't included them in the figure.    
 
 Note that the fundamental domains $B_{ji}^k$ are analogous to the fundamental domains $A_{j0}$ whereas the unions of fundamental domains $\mathbf{B}_{ji}$ are analogous to the unions of fundamental domains $\mathbf{A}_{\bj_{n-1}0}$.   We will preserve this analogy and notation in the inductive definition below.

 \item We inductively define the domains  $\mathbf{B}_{\bj_n}=R_{j}(\mathbf{B}_{\bj_{n-1}})$.  Thus $\bj_n$ has the form $j\bj_{n-2}i$, $i \neq 0$.  
For each  $j$, if $i>0$,  the $\mathbf{B}_{\bj_n}$ cluster at the prepole $p_{ \bj_{n-2} i}$ and  if $i<0$ they cluster at the prepole $p_{ \bj_{n-2}(i+1)}$.  Each has an outer biinfinite boundary curve, $\gamma_{\bj_n} $, both of whose endpoints are at the same pole and an interior  curve $l_{\bj_n}$, joining the pole $p_{\bj_{n-2}i}$ to $q_{\bj_n}$.   Each of the domains  $\mathbf{B}_{\bj_n}$ is a union of fundamental domains $B_{\bj_n}^k$.   
  
 \item   We remark that the admissibility condition for the $\mathbf{A}_{\bj_n}$ is that the rightmost entry in the sequence  $\bj_n$  is  always zero while the condition for the $\mathbf{B}_{\bj_n}$ is that the rightmost entry is never zero.  The geometry of these regions is different.   The 
 $\mathbf{A}_{\bj_n}$  have infinitely many vertices at infinitely many distinct prepoles. Each interior fundamental domain has vertices  at three or four distinct prepoles. The 
 $\mathbf{B}_{\bj_n}$, on the other hand,  have only one vertex at a single prepole, or infinity if $n=1$, where all the boundary curves meet.  The fundamental domains contained in them are annuli, the outermost of which has a prepole on its boundary.  
 \end{enumerate}

\subsubsection{ The Coordinates }\label{modelcoords}


 We  extend the map $\phi_0$ defined above to  all of $K_0$ by analytic continuation.  Continuing  across the boundary of $\Delta$ we have
$$\phi_0^{-1} ( \{ \zeta \, | \, r_0  \leq  |\zeta |  \leq r_0/\rho \})=\overline{A_0}.$$
We use the closure here to include the boundary curves.   We extend the map to all of $K_0 \setminus \{ \la_0\}$ in the obvious way: if $z$ is in a fundamental domain $Q^{-n}(A_0)$ for any $n$ inverse branches, we set $\rho^n\phi_0(z)=\phi_0(Q^n(z))$.

We define  coordinates for $K_0$ in terms of local coordinates in the fundamental domains described above.
\begin{enumerate}
\item The curves $\phi_0^{-1}(|\zeta|=r)$  are the {\em level curves of level $r$} in $K_0$.   The level of $\gamma^*$ is $r_0$ and the level of $\gamma_j$ is $r_0/\rho$.

 The outer boundary of   each $\mathbf{B}_{{\mathbf j}_n}$ has level $r_0/\rho^n$;  passing through the interior fundamental domains, the levels decrease to zero by powers of $\rho^k$.  
In each $\mathbf{A}_{{\mathbf j}_n}$ the levels  go from $r_0/\rho^{n-1}$ to  $r_0/ \rho^n$ and the same is true in the interior fundamental domains. 

\item
   The {\em gradient curves} are preimages of the radii  $\theta= \theta_0$  for a fixed $\theta_0 \in \RR$ under the map $\phi_0^{-1}(r e^{i \theta})$.  For example, each of the fundamental domains $A_{{\mathbf j}_n}$ has four boundary curves; one pair of opposite curves,  labeled with  $\gamma$'s and shown in red in figure~\ref{ModelTess}, are level curves and the other pair of opposite curves,  labeled with  $l$'s and shown in green in figure~\ref{ModelTess},  are gradient curves  along which the level rises from  $r_0/\rho^{n-1}$ to level $r_0/\rho^{n}$.

 \item The level and gradient curves  define a set of {\em local coordinates  in  $K_0 $}:   the preimages of the circles and radii in $\DD_0$ under the extension of $\phi_0$ pull back to each fundamental domain.
 We denote the coordinates of the point $z \in K_0$ by
 \[ z=(X_{{\mathbf j}_n},r, \theta+\pi(n-1)) \]
  where $ X_{{\mathbf j}_n}=\overline{A_{{\mathbf j}_n}}$ if $\bj_n = j_1 \ldots j_{n-1}0$ and $X_{\bj_n}= \overline{B_{{\mathbf j}_n}^k}$ if $\bj_n = j_1 \ldots j_{n-1}j$, $j \neq 0$.  Because $k$ can be read off from  $ r \in [0, \infty)$, it is enough to write $X_{\bj_n}$.   We let $  \theta \in [-\pi, \pi)  $ and 
  note that because the inverse branches are defined as one sided  limits on the boundaries of the fundamental domains, the $\theta$ coordinate varies continuously across common boundaries.
  \end{enumerate}

\begin{thm} \label{coords in K0} Every point $z \in K_0$,  $z \neq \la_0$, is in either a unique $A_{\bj_n}$ or a unique $B_{\bj_n}^k$ or on the boundary of two such domains: either the common  boundary of some $A_{\bj_n}$ and some $A_{\bj_{n+1}}$  where $\bj_{n+1}=\bj_n 0$,  or the common boundary of an $A_{\bj_n}$  and a $B_{\bj_{n-1}}^0$ where
$\bj_{n}=\bj_{n-1}j_{n}$, $j_n \neq 0$ or the common boundary of a $B_{\bj_{n-1}}^k$ and a $B_{\bj_{n-1}}^{k+1}$.
 \end{thm}

 \begin{proof}  Let $z \in K_0$,  $z \neq \la_0$.  Since  every such $z$ is attracted to $q_0$  and has infinitely many preimages, there is an $m$ such that $Q^m(z) \in \Delta \cup (\partial\Delta \setminus \{\la_0 \})$.   If $\zeta=Q^m(z) \in \Delta$, there is a unique set of preimages $R_j$ such that $z=R_{{\mathbf j}_m}(\zeta)$.  If $\zeta \in  (\partial\Delta \setminus \{\la_0 \})$ the  $\theta$ coordinate is defined as a one sided limit for each preimage, and the limits agree on the boundary curves.
 \end{proof}

 \subsubsection{The Tree in $K_0$}\label{combtree}

 In this section we use the domains $\mathbf{A}_{\bj_n}$ which are unions of the fundamental domains $A_{\bj_n 0} $ to define a   tree in  $K_0$.   For readability, we will say a level curve of level $r_0/\rho^{n}$ has level $n$.   
 
 Among the boundary curves of the union $\mathbf{A}_{\bj_n}$, there is a distinguished boundary curve of level $n$; it is a  preimage of $\gamma_0$ under the map $R_{\bj_n}$.  The remaining infinitely many boundary curves of the union have level ${n+1}$ and are images under the maps $R_{\bj_nj}$.   Note that the non-distinguished boundary curves  of $\mathbf{A}_{\bj_n}$  of level $n$ are also boundaries of some $\mathbf{B}_{\bj_n}$.    At level $n=1$ we will fix a root node on $\gamma_0$.   On each    $\gamma_{\bj_n0}$ of level $n$,  $n>1$, we will define  the preimage of the root to be a node of the tree. We call these interior nodes.   We will also put  nodes at every prepole of every order on the tree.

The children of the interior node of level $n$ are the interior nodes of level $n+1$,  a prepole of level $n$ and infinitely many prepoles of level $n+1$.   We will define branches of the tree that connect a node to each of its children.  The nodes that are prepoles have no children and are called leaves of the tree.  Each interior node has only one parent and each prepole node has two parents.    Paths through the tree start at the root and consist of a connected set of branches joining nodes in the tree.  Some will be finite, ending at  leaves and others will be infinite.  
 
%

  \begin{figure}
  \centering
  \includegraphics[width=5in]{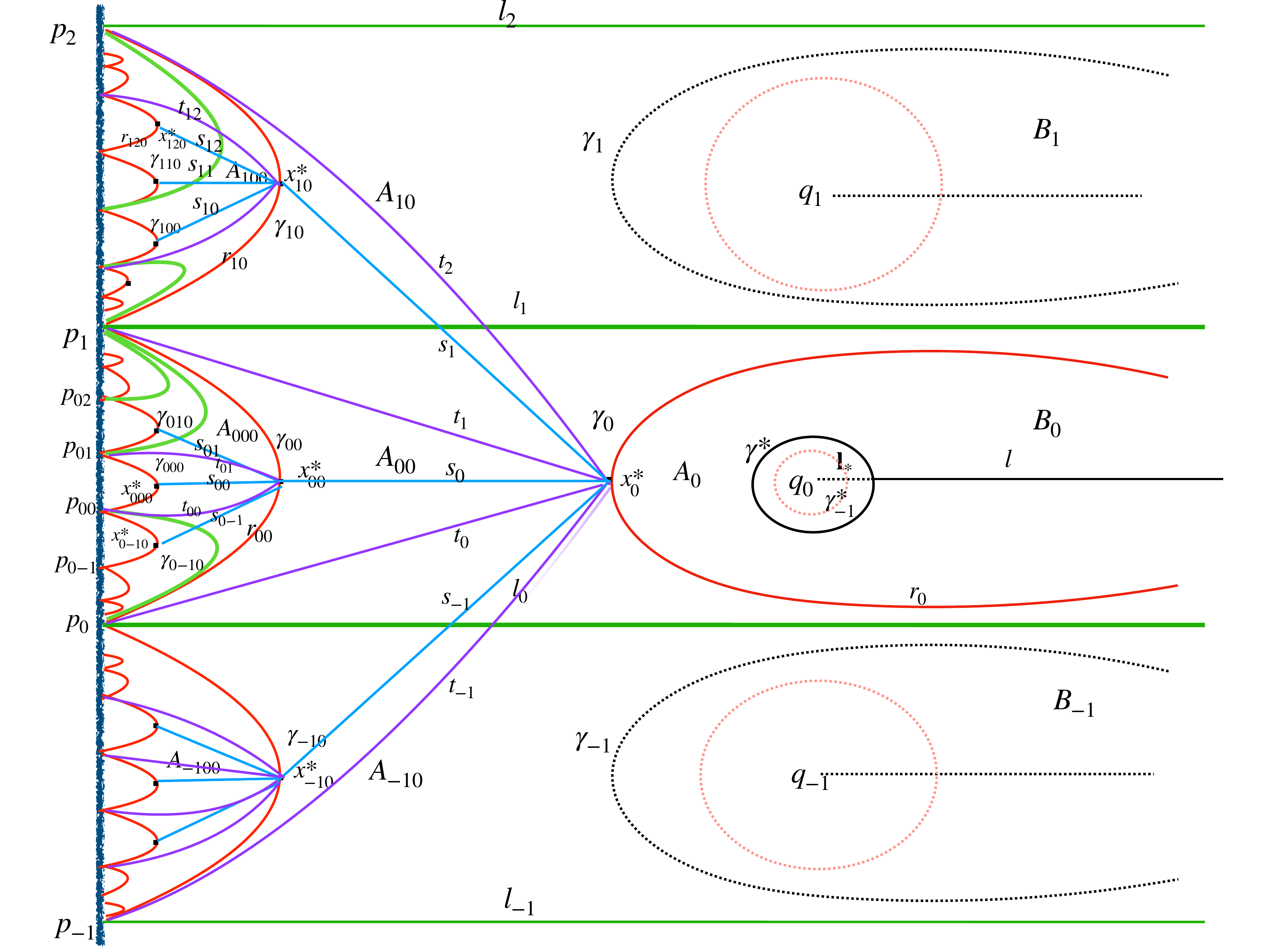}
  \caption{The fundamental domains of the model space with the tree. }\label{ModelA}
\end{figure}

 For the explicit construction of the tree, (see  figure~\ref{ModelA}),   
  fix a point  $x_0^*$  on  the boundary curve $\gamma_0$ of the domain $A_0$.  It has level $1$ and is the  first node, or root of the tree.      For each  $n>1$,  and each $\bj_n=\bj_{n-1}0$, the interior nodes of level $n$ are defined as the points $x_{\bj_{n-1}0}^*=R_{\bj_n}(x_0^*)$.    The prepoles of all orders are leaves of the tree. 
  
  The first step of the construction is to define a  tree $T^*$.    Join the root $x_0^*$ to each of the nodes $x_{j 0}^*$, $j \in \ZZ$, by a branch $s_{j}$ and join  it to the pole  $p_{j}$ by a branch $t_{j}$.    Also define the segment  $r_0$ of $\gamma_0$ from $x_0^*$ to infinity,   asymptotic to the line $l_0$ as a branch   joining the root to the ``prepole'' infinity.  The root with these branches connected to the leaves 
  is a small tree $T^*$ contained in $\mathbf{A}_0$.   Since the domain  $\mathbf{A}_0$ admits a hyperbolic metric,  we may choose $s_{j}$   as  geodesics in the hyperbolic metric and  take the $t_j$ and $r_0$ along the level curves.  The hyperbolic lengths of the $s_j$  go to infinity with $|j|$ while the lengths of the $t_j$ and $r_0$ are always infinite. 

Now define a small tree at each node $x_{\bj_{n}0}^*$,  and  contained in $\mathbf{A}_{\bj_n 0}$,  by $T_{\bj_n}^*=R_{\bj_n}(T^*)$.  Note  that as often happens,  the spatial relationship of the nodes in the tree is dual to the dynamic relationship.  The nodes $x^*_{j_1 k 0}$, $k \in \ZZ$, are children of the node  $x^*_{{j_1}0}$, whereas the nodes $x^*_{{k j_1}0}$ which are preimage of $x_{j_10}$ are not.   Thus the children of the parent node $x_{\bj_{n}0}^*$ are the nodes $x_{\bj_{n} k0}^*$ and not the nodes $x_{k \bj_{n}0}^*$.     The tree  has branches $s_{\bj_n }$ joining  the parent node $x_{\bj_{n}0}^*$ to its interior children $x^*_{j_1 k 0}$, a branch $r_{\bj_n0}$ joining it to the prepole $p_{\bj_n}$  and branches $t_{\bj_n k }$   joining it to the  prepoles $p_{\bj_{n}k}$.    
  Because the $R_j$ are biholomorphic, the   hyperbolic lengths of   the branches are preserved.  

  Finally, joining all these small trees together,  we obtain the full tree
\[ T_{\infty}^*= \bigcup_n \cup_{\bj_n}(T_{\bj_n}). \]

If ${\bjinf}$ is periodic with period $n$, there is a sequence $\bj_n= j_1 \ldots j_n$ such that ${\bjinf}=\bj_n\bj_n\bj_n\ldots $.  By  the periodicity,   $j_{1+n}=j_1$, so that  the node  $x^*_{j_1\bj_n 0}$ is a direct descendant of the node   $x^*_{j_1 0}$. 
 This means that there is a path $\tau_{\bj_n}=s_{j_1} \cup s_{j_1 j_2} \cup \ldots s_{j_1 \ldots j_{n-1}}$ from the root to the node.   It has finite hyperbolic length.    
 Note, however, as we remarked above,  applying $R_{j_i}$ to successively obtain the nodes $x^*_{j_n 0}, x^*_{ j_{n-1} j_n 0},   \ldots,   x^*_{j_2 \ldots j_{n-1} j_{n} 0}$,  yields a collection of nodes that  are not direct descendants of  $x^*_{j_1 0}$.

The path 
\[  \tau_{\bj_n \bj_n}=R_{\bj_n}(\tau_{\bj_n})= R_{\bj_n}(s_{j_1}) \cup R_{\bj_n}(s_{j_1 j_2}) \cup \ldots R_{\bj_n}( s_{j_1 \ldots j_{n-1}}) \] 
joins $x^*_{\bj_n 0}$  to $x^*_{\bj_n\bj_n 0}$ and has the same hyperbolic length as $\tau_{\bj_n}$.
Iterating,  we obtain a path $\tau_{\bjinf}$ of infinite length that is invariant under $R_{\bj_n}$.

 All but the first segment of this path are separated from the root by the level curve $\gamma_{j_10}$ so any accumulation point of the path lies in the Julia set between the poles $p_{j_1}$ and $p_{j_1 +1}$.  Thus  the path  
  remains inside a compact domain bounded by $\gamma_{j_1 0}$ and the Julia set boundary of $K_0$.   This implies that the Euclidean lengths of subpaths making up $\tau_{\bj}$ tend to zero, and since $Q$ is hyperbolic, that the full path has  a unique endpoint.

 Thus if $\bjinf$ is periodic of period $n$,   the path in the tree with   unique endpoint $z_{\bjinf}$  is invariant under $Q^n$ and  corresponds to a repelling periodic point of $Q$ in $J_0$ whose  combinatorics agree with those defined in Proposition~\ref{Prop J0}.

If $\bjinf$ is preperiodic, $\bj=\bj_m\bj_n\bj_n\bj_n  \ldots$, we can construct a periodic  infinite subpath of $\tau_{\bjinf}$ beginning at $x_{\bj_m}^*=R_{\bj_m}(x_0^*)$, instead of the root, so that it is invariant under $R_{\bj_n}$.   The argument above shows it also has a unique preperiodic endpoint.


 \section{The Shift Locus}\label{shift}

 In the shift locus $\cals$, both asymptotic values are attracted to the origin.    If $\la \in \cals$,  we can define a linearizing map $\phi_{\la}$ from the attractive basin $A_{\la}$ of the origin to the disk $\DD_0$ that is injective on a neighborhood  $O_{\la}$ of the origin.   Neither  $\la$ nor $\mu$ lies in $O_{\la}$ and one or both lie on $\partial O_{\la}$.

 We divide $\cals$ into disjoint subsets as follows
\[  \cals_{\la}^0 = \{ \la \in  \cals \, | \, \, \mu \in \partial O_{\la}, \, \la \not\in \partial O_{\la} \}, \]
 \[  \cals_{\mu}^0 = \{ \la \in \cals \, | \, \,  \la \in \partial O_{\la}, \,  \mu \not\in \partial O_{\la}\},  \mbox{  and } \]
 \[ \cals_* = \{ \la \in  \cals \, | \, \, \la \in \partial O_{\la}, \, \mu \in \partial O_{\la}\}, \]
  We normalize the map so
   that if $z_0$ equal to the asymptotic value on the boundary and   $z \in O_{\la}$ then
 \[ \phi_{\la}(0)=0,  \phi_{\la}(z_0)=\phi_0(\la_0)=r_0,
  \mbox{ and  }                      \phi_{\la}(f_{\la}(z)) = \rho\phi_{\la}(z). \]
  Note that this normalization agrees with our normalization of $\phi_0$, the linearizing map for the model $Q$;  that is, both map the asymptotic value to the point $r_0$ on the real axis.  

We restrict our discussion here to $\cals_{\la}^0$ but there is a comparable discussion for $\cals_{\mu}^0$.

\subsection{Coordinates in the dynamic plane $A_{\la}$}
 The scheme we defined above for tesselating the attracting basin $K_0$ of $Q=f_{\la_0}$ works equally well  in a subdomain of  the attractive basin of zero, $A_{\la}$, for $\la \in \cals_{\la}^0$.
 
   \begin{thm}\label{la basin coords}  Given $\la \in \cals_{\la}^0$,  there is a   coordinate structure defined on a subdomain $\hat{A}_{\la}$ of $A_{\la}$.  More precisely,  there is an integer $N$ such that the basin of the origin of $f_{\la}$, $A_{\la}$, contains a subdomain $\hat{A}_{\la}$ tesselated by fundamental domains $\alpha_{\la,{\mathbf j}_{n-1}0}$ and $\beta_{\la,{\mathbf j}_{n-1}i}^k$, $i \neq 0$, $k  \geq 0$ and $n \leq N$.   The boundary curves of these regions are level  and gradient curves defined by  using a normalized linearizing function $\phi_{\la}$ near the origin, and pulling back a radius and circle containing $\phi_{\la}(\mu)$ to $A_{\la}$.  The geometric properties of the $\alpha_{\la,{\mathbf j}_{n-1}0}$ and $\beta_{\la,{\mathbf j}_{n-1}i}^k$ are analogous to those of the  fundamental domains $A_{\bj_{n-1}0,\la}$ and $B_{\bj_{n-1}i,\la}^k$ domains in $K_0$.  The coordinates in $\hat{A}_{\la}$ are $(\sigma_{\bj_n}, r, \theta +\pi(n-1))$ where $r \in [0, \infty), \theta \in [-\pi, \pi)$ and   $\sigma_{\bj_n}=\sigma_{\bj_{n-1}i}$ stands for $\alpha_{\la,{\mathbf j}_{n-1}0}$ if $i=0$ and $\beta_{\la,{\mathbf j}_{n-1}i}^k$ for some $k$ depending on $r$ for $i \neq 0$. 
 \end{thm}

\begin{proof}   For $\la \in \cals_{\la}^0$,  the attractive basin of zero, $A_{\la}$ contains both asymptotic values.    By definition, the linearizing map  $\phi_{\la}$ is a  homeomorphism from $O_{\la}$, an open neighborhood of the origin with $\mu$ on its boundary, onto the disk $\DD_0$ of radius $r_0$.   It is normalized so that $\phi_{\la}(0)=0$ and $\phi_{\la}(\mu)=\phi_0(\la_0)=r_0$.   
  Extending $\phi_{\la}$ by analytic continuation,
   the analogues of the domains $A_{{\mathbf j}_n}$ and $B_{{\mathbf j}_n}^k$ can be defined as in section~\ref{modelcoords} until, for some $n=N$, one of them contains $\la$.    That is, 
 $r_0/\rho^{N-1}<|\phi_\lambda(\lambda)|\leq r_0/\rho^N$.  The level and gradient curves are well defined in these domains by the  branches of $\phi_{\la}^{-1}$.

To this end, we use  the map $\xi_{\la} = \phi_{0}^{-1} \circ \phi_{\la}:  O_{\la} \rightarrow O_{\la_0}$ that we defined In \cite{CJK2}.  
The inverse map,  $\xi_{\la}^{-1}$, extends, as a homeomorphism from a subset, $K_0(\la)$ to a largest subset $\hat{A}_{\la}$ of $A_{\la}$ that contains $\la$.   
Therefore  $\xi_{\la}^{-1}$ is defined on the fundamental domains
    $A_{{\mathbf j}_n}$ and $B_{{\mathbf j}_n}^k$ tesselating $K_0$, $k\geq 0$, $n \leq N$,  where  $N$ is the largest integer such that $|\phi_{\la}(\la)| \leq r_{0}/\rho^N$.

 Set $\alpha_{\la,{\mathbf j}_n}= \xi_{\la}^{-1}(A_{{\mathbf j}_n})$ and $\beta_{\la,{\mathbf j}_n}^k= \xi_{\la}^{-1}(B_{{\mathbf j}_n}^k)$.  The boundary curves of these domains are, by definition,  level and gradient curves for $A_{\la}$ and the relative levels correspond, via the map $\xi_{\la}^{-1}$, to the levels of  the corresponding curves in $K_0$. Moreover, since  we can define  inverse branches $R_{\la,j}$ of $f_{\la}$ on $\hat{A}_{\la}$ using the relation  $R_{\la,j}= \xi_{\la} \circ R_j \circ \xi_{\la}^{-1}$, the indexing is consistent with the model.   Thus  we obtain a coordinate   $(\sigma_{\bj_n}, r, \theta +\pi(n-1))$ for $\hat{A}_{\la}$.  
 \end{proof}
 
We can also use the map $\xi_{\la}^{-1}$ to obtain a tree in $\hat{A}_{\la}$, $T^*_{\la}=\xi_{\la}^{-1}(T^*_{\infty} \cap K_0^{\la})$.   The root of this tree is $x_{\la}^*=\xi_{\la}^{-1}(x_0^*)$.  Its nodes are defined similarly.   Note that some of images of  infinite paths in $T_{\infty}$ are truncated and so are  finite in $T^*_{\la}$.

 \medskip
 \subsection{ Coordinates in $\cals_{\la}^0$ }

In \cite{CJK2} we proved

\begin{thm} There is a homeomorphism $E: \cals_{\la}^0 \rightarrow K_0 \setminus \overline{\Delta}$.  Thus $\cals_{\la}^{0}$ is homeomorphic to an annulus $\AA$.  If $I$ is the inner boundary of $\AA$,  $I=\partial \Delta$ and $E^{-1}$ extends continuously to all points on $I$ except $\la_0$. The point  $E^{-1}(\la_0)$ corresponds to the parameter singularity $\la=0$ on the inner boundary of $\cals_{\la}^0$ where the function $f_{\la}$ is not defined.  The outer boundary of $E^{-1}(\AA)$ is contained in $\partial \calm_{\la}$ and contains all the virtual centers.
\end{thm}

To define the map $E$ we use the maps $\phi_0$,  $\phi_{\la}$  and  $\xi_{\la} = \phi_{0}^{-1} \circ \phi_{\la}$,  and 
 set $E(\la)=\xi_{\la}(\la)$.  It is not difficult to prove the map is injective.
 We then prove the map is a homeomorphism by the following construction: to each $\zeta \in K_0 \setminus \Delta$,   $\zeta\neq \lambda_0$,  we inductively construct a sequence of covering spaces of $K_0 \setminus  \{\la_0, \zeta\}$ and corresponding covering maps.  Using quasiconformal surgery, we prove the direct limit of this process is a map in $\cals_{\la}^0$.
  \medskip

The inverse holomorphic homeomorphism  $E^{-1}$  can be used to   define a tesselation and   coordinates  in $\cals_{\la}^0$.  For each sequence $\mathbf{j_n}$,  define $\cala_{\mathbf{j}_n}=E^{-1}(A_{\bj_{{n-1}0}})$ and $\calb_{\mathbf{j}_n}^k=E^{-1}(B_{\bj_{n-1}i}^k)$, $i \neq 0$, $k \geq 0$.  This identification  immediately gives us, (see Figure~\ref{ParamCoords}),

 \begin{figure}
  \centering
  \includegraphics[width=5in]{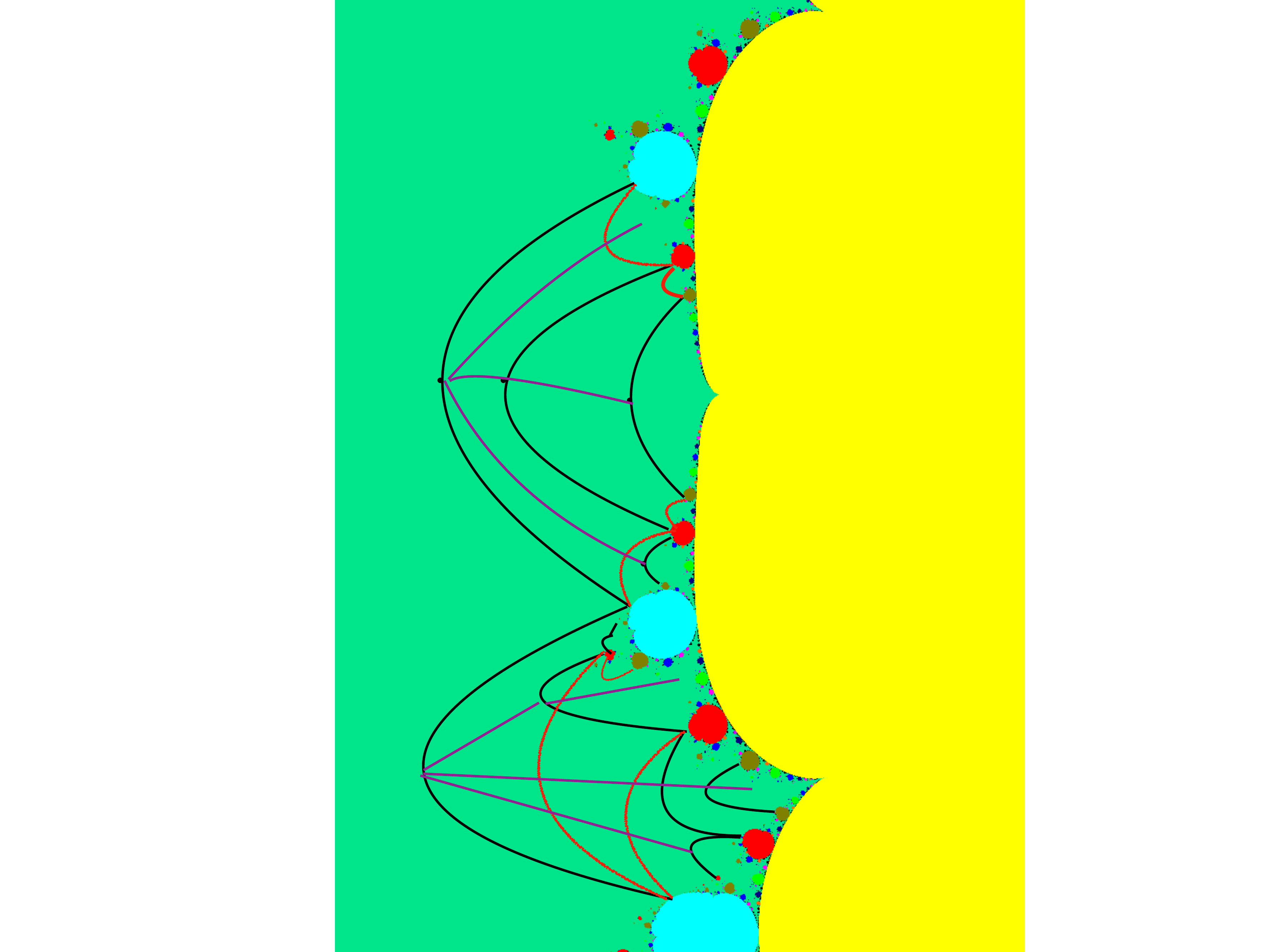}
  \caption{ The parameter space with some of the level curves in black, gradient curves in red and paths from the tree in purple.}\label{ParamCoords}
\end{figure}

\begin{thm}\label{coords in cals}  Each point $\la \in \cals_{\la}^0$ has a unique coordinate  $\la=(\mathcal{X}_{\mathbf{j}_n},r, \theta)$  where $\mathcal{X}_{\mathbf{j}_n}$ is either $\cala_{\mathbf{j}_n}$ or $\calb_{\mathbf{j}_n}^k$, $r \in [0, \infty), \theta = t + (n-1)\pi \in \RR$, $0\leq t<\pi$. \end{thm}

\section{The boundaries of $K_0$ and $\calm_{\la}$}\label{boundaries}

We are now ready to prove our main result.

  \begin{thm}\label{vcbdy}The injective holomorphic  map $E: \cals_{\la}^0 \rightarrow K_0 \setminus \overline{\Delta} $ extends continuously to the virtual centers of $\partial\cals_{\la}^0$   and maps them to prepoles of $Q$ with the same itinerary. \end{thm}

 \begin{proof}
Fix a finite sequence $\bj_n$ and let $\tau(t)$, $t \in [0,1]$ be a  path in the tree $T_{\infty}$ that ends at the prepole $p_{\bj_n}$. That is, it passes from the root  to the node $x^*_{\bj_n 0}$ and its last branch $r_{\bj_n0}$ goes from $x^*_{\bj_n 0}$ to  the prepole $p_{\bj_n}$ along the level curve $\gamma_{\bj_n 0}$.  The map  $E^{-1}$ then maps $\tau(t), t \in [0,1)$, to a path $\la(t) \in \cals_{\la}^0$.  

 We claim that the accumulation set of $\la(t)$ as $t$ goes to $1$ is a single point and that this point is a virtual cycle parameter.

 Let $\la_{\infty} \in \partial\cals_{\la}^0$ be an accumulation point of $\la(t)$ as $t$ goes to $1$ and let  $t_m$ be  sequence tending to $1$ such that $\la_m=\la(t_m)$ has limit $\la_{\infty}$.   Since we are only interested in $\tau(t)$ for $t$ close to $1$, we may assume all the points $\tau_m=\tau(t_m)$ belong to the last edge $t_{\bj_n}$.

 Note that the attractive basins $A_{\la}$ of $f_{\la}$  and  the boundary curves defining their tesselations  by fundamental domains  $\alpha_{\la,{\mathbf j}_{n-1}0}$ and $\beta_{\la,{\mathbf j}_{n-1}i}^k$, $i \neq 0$, in $\hat{A}_{\la} \subset A_{\la}$ all move holomorphically with $\la$.  
 
  In particular, the unions of these domains  $\boldsymbol{\alpha}_{\la,{\mathbf j}_{n-1}0} = \xi_{\la}^{-1}(\mathbf{A}_{\bj_{n-1}0})$ and $\boldsymbol{\beta}_{\la,{\mathbf j}_{n-1}i}=\xi_{\la}^{-1}(\mathbf{B}_{\bj_{n-1}i})$, $i \neq 0$, and their prepole boundary points, including the prepole $p_{\la,\bj_n}$, move holomorphically.   Thus, as $m$ goes to infinity, the functions $f_{\la_m}$ converge to $f_{\la_{\infty}}$ and the prepoles $p_{\la_m,\bj}$ converge to a prepole $p_{\la_{\infty,\mathbf{j}_n}}$ of $f_{\la_{\infty}}$.  Moreover,  $\tau_m$ is on a level curve in $K_0$,  and the images under $\xi_{\la}$ of the level curves in $K_0$ are level curves in $A_{\la}$,  so each $\la_m = \xi_{\la_m}(\tau_m)$ is on a level curve of the same level.  The level curves  in $A_{\la_m}$ containing $\la_m$ have endpoints at prepoles so that $\lim_{m \to \infty} \xi_{\la_m}(\tau_m)=p_{\la_\infty, \bj_n}$. Therefore either $\la_{\infty} \in A_{\la_{\infty}}$  so that  $\la_{\infty} \in \cals_{\la}^0$,  or
 \[ | \la_m - p_{\la_{\infty,\mathbf{j}_n}}| \leq |\la_m -p_{\la_m,\bj_n}| + |p_{\la_m,\bj_n}- p_{\la_{\infty,\mathbf{j}_n}}| \rightarrow 0  \text{ as } m\to \infty.\]
 The first possibility cannot happen since we assumed $\la_{\infty} \not\in \cals_{\la}^0$.  The second says that $\la_{\infty}$ is a virtual cycle parameter.   Since the sequence $t_m$ was arbitrary and the prepoles of any given order form a discrete set, the limit is independent of the sequence and thus unique.
 \end{proof}

We turn now to the periodic points in the Julia set of $Q$ and show that the map $E^{-1}$ extends to them.
The proof is similar to the above.

\begin{thm}\label{parab bdy pts}  The injective holomorphic map $E^{-1}: K_0\setminus \overline{\Delta} \rightarrow \cals_\la^0$ extends continuously to the repelling periodic points in $\partial K_0$ and maps them to points in $\la \in \partial\cals_{\la}^0$ for which $f_{\la}$ has a parabolic cycle of the same period.
\end{thm}

\begin{proof}  Let $\bjinf=\bj_n\bj_n \ldots$ be a periodic infinite sequence and let $z_{\bjinf}\in \partial K_0$ be the repelling point of order $n$ in the Julia set of $Q$ corresponding to this sequence.  Let $\tau_{\bjinf}(t)$, $t \in [0,1)$  be the infinite   path in $T_{\infty}$ corresponding to the sequence.    It is invariant under $Q^n$ and, since $Q$ is hyperbolic, its endpoint in $J_0$ is well defined and is the  repelling periodic point $z_{\bj}$.

Let $\la(t)= E^{-1}(\tau_{\bjinf}(t))$.  We claim this path lands on $\partial\cals_{\la}^0$ as $t$ goes to $1$.   Let $\la_{\infty}$ be any point in the accumulation set of the path and  let  $t_m$ be a sequence tending to $1$ such that $\la_m=\la(t_m)$ has limit $\la_{\infty}$.


For each $m$,  there is an integer $k(m)$ such that if $\bj_{k(m)}$ is a truncation of the periodic sequence $\bjinf$ after $k(m)$ repetitions of $\bj_n$, $\la_m \in \cala_{\bj_{k(m)}} \subset \cals_{\la}^0$.  This means we also have  $\la_m \in \boldsymbol{\alpha}_{\la_m, \bj_{k(m)} }\subset \hat{A}_{\la_m}$ and $\xi_{\la_m}(\la_m) \in \mathbf{A}_{\bj_{k(m)}} \subset K_0$.  Let 
$\hat\tau_{\bj_{k(m)} }\in T_{\infty}$ be the  tree  $\tau_{\bjinf}$ in $K_0$ up to the node $x_{\bj_{k(m)} 0}^*$,  and  having as its  final branch,  $r_{\bj_{k(m)}}$, the level curve from the node to the prepole $p_{\bj_{k(m) }}$.  The last fundamental domain it passes through is $\mathbf{A}_{\bj_{k(m)}}$. 

Using the map $\xi_{\la}^{-1}$ we can pull back $\hat\tau_{\bj_{k(m)}}$ to a tree $\hat\tau_{\la_m, \bj_{k(m)}} \subset \hat{A}_{\la_m}$.   The last fundamental domain it passes through is $\boldsymbol{\alpha}_{\bj_{k(m)}}$ and this fundamental domain contains $\la_m$.   We can modify the branch of $\hat\tau_{\la_m,\bj_{k(m)}}$ in $\boldsymbol{\alpha}_{\bj_{k(m)}}$ so that it passes through $\la_m$.  We will do this, and by abuse of notation, denote  the modified tree by $\hat\tau_{\la_m,\bj_{k(m)}}$ again.

Everything is holomorphic in $\la$,  and as $k$ goes to infinity, $\bj_{k(m)} \to \bjinf$, so the prepoles $p_{\bj_{k(m)}} \in J_0$ tend to the repelling periodic point $z_{\bjinf} \in J_0$.   It follows from the sequence topology that the  prepoles $p_{\bj_{k(m)}, \la_m}$ tend to the repelling periodic point $z_{\la_m,\bjinf }$ and the repelling periodic points  $z_{ \la_m.\bjinf}$ tend to $z_{\la_{\infty},{\bjinf}}$.  This must be  a repelling or parabolic periodic point of $f_{\la_{\infty}}$.   It cannot be the point $\la_{\infty}$ because  an asymptotic value of $f_{\la}$ cannot be periodic.

We claim that $z_{ \la_{\infty},{\bjinf} }$  must be a parabolic periodic point of $f_{\la_{\infty}}$.    We first show it must be a neutral periodic point. Suppose  $z_{\la_{\infty},\bjinf}$ is a repelling periodic point.  Then there is a neighborhood $U$ containing $\la_{\infty}$ such that $z_{ \la, \bjinf}$ is repelling for all $\la \in U \cap \cals_{\la}^0$.  In particular, it contains $\la_m$ for large enough $m$.   Then for each such  $m$, we  modify $\hat\tau_{\la_m,\bj_{k(m)}}$ by changing its last branch.  We do this by replacing $r_{\bj_{k(m)}} \in \hat\tau_{\bj_{k(m)}}$ with a path in $K_0$, monotonic increasing with respect to  level, and ending at $z_{\bjinf}$.  We call the result $\hat{\hat\tau}_{\bj_{k(m)}}$.  Then  $\xi_{\la_m}^{-1}( \hat{\hat\tau}_{\bj_{k(m)}})$ is a path in  $A_{\la_m}$ ending at the 
    repelling periodic point  $z_{ \la_m, \bjinf}$.   Again, as $m$ goes to infinity, the $\hat{\hat\tau}_{\bj_{k(m)}}$'s  converge to a path in $\CC$ with endpoint $z_{\la_{\infty}, \bjinf}$, the periodic endpoint of $f_{\la_{\infty}}$.   If $\la_m$ were a point on $\xi_{\la_m}^{-1}( \hat{\hat\tau}_{\bj_{k(m)}})$, the $\la_m$'s would either converge to a point in $A_{\la_{\infty}}$ or to a repelling periodic point of $f_{\la_{\infty}}$.   The first case cannot happen since $\la_{\infty}$ is not an interior point of $\cals_{\la}^0$ and the second cannot happen since $\la_{\infty}$ cannot be periodic.

Therefore the fixed point  $z_{ \la_{\infty},{\bjinf} }$  is neutral.   A standard application of the Snail Lemma \cite[p.154]{M2}, shows that it must be parabolic.
 \end{proof}

As a corollary of the proof of this theorem, it follows  that the injective homeomorphism $E^{-1}: K_{0}\setminus \overline{\Delta}\to {\mathcal S}^{0}_{\lambda}$  extends continuously to the eventually periodic points in $\partial K_{0}$ and maps them to points in $\lambda\in \partial {\mathcal S}_{\lambda}^{0}$.  Because $\la_{\infty}$ does not belong to the cycle containing  $z_{\bjinf, \la_{\infty}}$, but maps onto it in finitely many steps, and $\la_{\infty}$ does belong to the Julia set, the cycle is repelling.
This, together with Theorem~\ref{vcbdy} and Theorem~\ref{parab bdy pts},
 completes the  proof of the Main Theorem in the introduction.

\end{document}